\documentclass[11pt, a4paper, reqno]{amsart}
\usepackage{amsmath,amssymb,amsthm,soul}
\usepackage{xcolor,textcomp}
\usepackage{graphicx, epstopdf}
\usepackage{braket,amsfonts}
\usepackage[T1]{fontenc}
\usepackage[utf8]{inputenc}
\usepackage{mathtools}
\usepackage{mwe}
\usepackage{cancel}
\usepackage[top=2.9cm,bottom=2.9cm,left=3cm,right=3cm]{geometry}
\usepackage{hyperref}
\hypersetup{colorlinks=true,linkcolor=red,citecolor=red,filecolor=magenta,urlcolor=blue}
\usepackage[capitalise]{cleveref}
\usepackage{comment}
\usepackage{tikz}
\usepackage{pgfplots}
\usepackage{stmaryrd}
\usepackage[sort,nocompress]{cite}
\numberwithin{equation}{section}

\newcommand{\spann}{\operatorname{span}}

\newcommand{\meas}{\operatorname{meas}}
\newcommand{\rank}{\operatorname{rank}}
\newcommand{\curl}{\operatorname{curl}}
\newcommand{\rico}{\operatorname{rico}}
\newcommand{\grad}{\operatorname{grad}}
\newcommand{\co}{\operatorname{co}}
\newcommand{\intt}{\operatorname{int}}
\newcommand{\aff}{\operatorname{aff}}
\newcommand{\rbd}{\operatorname{rbd}}
\newcommand{\ri}{\operatorname{ri}}

\newtheorem{theorem}{Theorem}[section]
\newtheorem{lemma}[theorem]{Lemma}
\newtheorem{remark}[theorem]{Remark}

\title[Differential Inclusions involving the curl operator]{Differential Inclusions involving the curl operator}
\author[N. Nesha]{ Nurun Nesha$^{\dagger,1}$ }
\thanks{$^\dagger$Indian Institute of Science Education and Research Kolkata, Campus Road, Mohanpur, West Bengal 741246, India; $^1$nn16ip021@gmail.com}

\pgfplotsset{compat=1.18} 
\begin{document}

\begin{abstract}
In this article, we study the existence of $\eta\in W_0^{1,\infty}(\Omega;\mathbb R^n)$ satisfying $$\curl \eta\in E \textrm{ a.e. in }\Omega,$$
where $n\in \mathbb N, \Omega\subseteq \mathbb R^n$ is open, bounded and $E\subseteq \Lambda^2.$
\end{abstract}
	
	
\subjclass[2010]{}

 
\maketitle
	
\tableofcontents
	
\section{Introduction and Main Results}
In this paper, we study the following differential inclusion problem 
\begin{equation}
\begin{aligned}
&\curl \eta \in E \ \ \textrm{ a.e. in } \Omega\\
&\textrm{ and } \int\limits_\Omega \eta \neq 0
\end{aligned}
\end{equation}\label{introeq1}
where $\Omega\subseteq \mathbb R^n$ is open, bounded,  $E\subseteq \Lambda^2(\mathbb R^n)$, and $n\geq 4$. This problem has been studied in Bandyopadhyay-Barroso-Dacorogna-Matias \cite{bandyopadhyay2007differential} and Bandyopadhyay-Dacorogna-Kneuss \cite{bandyopadhyay2015some} in the lower dimensional cases, namely when $\dim \spann E=n-1$ when $n\geq 3,$ $\dim \spann E=3,$ when $n=3.$ In this article, we investigate the case when $\dim \spann E\geq n.$ 
The most fundamental case is, of course, the gradient case which has received notable attention, in particular, by Bressan-Flores \cite{BressanFlores}, Cellina \cite{cellina1993minima,cellina1993minimanecessary}, Dacorogna-Marcellini \cite{DacorognaMarcellini} and Friesecke \cite{Friesecke}.  An extensive study has been done in \cite{dacorogna2005nonconvex} on this topic.
We prove a few  existence as well as non-existence results in this regard. Our main result is the following which we will prove in section $4.$
\begin{theorem}
    Let $n\in \mathbb N, n\geq 5$ and $E\subseteq \Lambda^2(\mathbb R^n)$ be such that $$\omega \wedge \omega'=0 \textrm{ for all }\ \omega ,\omega'\in E.$$
    Let $\Omega \subseteq \mathbb R^n$ be  an open, bounded set.  Then there exists $\eta \in W_0^{1,\infty}(\Omega;\mathbb R^n)$ such that
    \begin{align} \label{16/12/2022 1`}
     & \curl \eta \in E \textrm{ a.e. in } \Omega  \nonumber \\
      & \meas\{x\in \Omega : \curl \eta (x)=e\}>0 \textrm{ for all } e\in E\\
      & \textrm{ and } \int\limits_{\Omega}\eta \neq 0 \nonumber
     \end{align}
     if and only if $0\in \rico E$ and $\dim \spann E=n-1.$
\end{theorem}
The next result is about  non-existence of solution when $\dim \spann E=n$ which we will discuss in section $3.$
\begin{theorem}
Let $n\in \mathbb N, n\geq 4$ and $\Omega \subseteq \mathbb R^n$ be open, bounded. Then the following differential inclusion problem 
\begin{equation*}
    \curl \eta \in E \textrm{ a.e.  in }\Omega \textrm{ and } \displaystyle \int\limits_{\Omega }\eta\neq 0
\end{equation*}
 has no solution $\eta \in W_{0}^{1,\infty}(\Omega;\mathbb R^n)$ if  $\dim \spann E=n$  and $\meas \{x\in \Omega: \curl \eta(x)=e\}>0$ for all $e\in E$.
\end{theorem}
In \cite{ball1987fine}, Ball-James considered two gradient problem and found that rank of the difference of two gradients is less than or equal to $1$ and in a similar way we can show that, in curl case the rank of the difference will be less than or equal to $2$. So, in section $4,$ we will see the following theorem under taking the constraint on the set $E$ that rank of difference of any two elements is less than or equal to $2.$
\begin{theorem}
Let $n\in \mathbb N$ and $\Omega \subseteq \mathbb R^n$ be open, bounded set. Let $E\subseteq \Lambda^2(\mathbb R^n)$ be such that $\rank[e-f]\leq 2$ for any $e,f\in E,$ in other words, there exist $x,y\in \mathbb R^n$ such that $e-f=x\wedge y.$ Then there does not exist any $\eta\in W_{0}^{1,\infty}(\Omega;\mathbb R^n)$ of the following problem 
\begin{equation*}
    \curl \eta \in E \ \ \textrm{ a.e. in } \Omega
\end{equation*}
if $\dim\spann E\geq n+1$ and $\meas\{x\in \Omega :\curl \eta(x)=e\}>0$ for all $e\in E.$
\end{theorem}
In section $5$, we will give one existence result of solution $\eta \in W_0^{1,\infty}(\Omega;\mathbb R^n)$ at dimension $(2n-3)$ for the following differential inclusion problem 
\begin{align*}
&\curl \eta\in E \textrm{ a.e. in }\Omega,\\
&\meas\{x\in \Omega:\curl \eta (x)=e\}>0 \textrm{ for all } e\in E.
\end{align*}
Finally, in section $6$, we show that
\begin{theorem}
Let $n\in \mathbb N,$ $1\leq k \leq n-3.$ Suppose $f:\mathbb R^n\to \Lambda^k(\mathbb R^n)$ be  continuous such that $f(0)\neq 0$ and $f(0)$ is $k$-divisible, i.e., $f(0)=c^1\wedge \ldots \wedge c^k$ for some $c^i\in \mathbb R^n\setminus \{0\}$ for  $i=1,\ldots, k.$ Then 
\begin{equation}\nonumber
\dim \spann \{x\wedge f(x): x\in \mathbb R^n\}\neq n-k+1.
\end{equation}

\end{theorem}
The results of differential inclusion problems can be applied, embracing the notions due to Cellina \cite{cellina1993minima,cellina1993minimanecessary} and Friesecke \cite{Friesecke}, to obtain solutions for a non-convex variational problem. In particular, according to \cite{bandyopadhyay2007differential}, one can show that: 
\begin{theorem}
    Let $\Omega\subseteq \mathbb R^n$ be a bounded, open set, $0\leq k\leq n-1$ and $$f:\Lambda^{k+1}(\mathbb R^n)\to \mathbb R_{+}$$
    be lower semi-continuous. Let 
 \begin{equation}\label{05/04/2024 eq 1} 
        \inf \left\{\displaystyle\int\limits_{\Omega}f(\textrm{d}\eta(x))\textrm{d}x:\eta\in W_{0}^{1,\infty}\left(\Omega;\Lambda^k(\mathbb R^n)\right)\right\}
    \end{equation}
    and $K=\left\{\xi \in \Lambda^{k+1}(\mathbb R^n):f^{**}(\xi)<f(\xi)\right\}$, where $f^{**}$ is the convex envelope of $f.$ Assume that $K$ is connected and $0\in K.$ If $K$ is bounded and $f^{**}$ is affine on $K$ then (\ref{05/04/2024 eq 1}) has a solution.
\end{theorem}
Many results as well as applications of differential inclusions can be observed in Dacorogna-Pisante \cite{PisanteDacorogna}, Dacorogna-Fonseca \cite{DacorognaFonseca}, Dacorogna-Marcellini \cite{DacorognaMarcellini}, Blasi-Pianigiani \cite{BlasiPianigiani}, Sil \cite{SwarnenduACV} and Sychev \cite{Sychev}.
\section{Notations}
We gather here some notations which will be used throughout this article. For more details on exterior algebra and differential forms see \cite{csato2011pullback} and for convex analysis see \cite{DirectBernard} or \cite{Rockafellar}.
\begin{itemize}
\item [$(1)$] Let $k,n$ be two integers. 
   \begin{itemize}
    \item [$\bullet$] We write $\Lambda^k(\mathbb R^n)$ (or simply $\Lambda^k$) to denote the vector space of all alternating $k$-linear maps $f:\underbrace{\mathbb R^n \times \cdots \times \mathbb R^n }_{\textrm{k-times}}\to \mathbb R.$ For $k=0,$ we set $\Lambda^0(\mathbb R^n)=\mathbb R.$ Note that $\Lambda^k(\mathbb R^n)=\{0\}$ for $k>n$ and, for $k\leq n,$ $\dim \left(\Lambda^k(\mathbb R^n)\right)=\binom{n}{k}.$
    \item [$\bullet$] $\wedge,$ $\lrcorner$, $\langle;\rangle$ and, respectively, $*$ denote the exterior product, interior product, the scalar product and, respectively, the Hodge star operator.
    \item [$\bullet$] For $b\in \Lambda^k,$ $\rank[b]$ denotes the rank of the exterior $k$-form $b.$
    \item [$\bullet$] If $\{e^1,\ldots, e^n\}$ is a basis of $\mathbb R^n$, then, identifying $\Lambda^1$ with $\mathbb R^n,$ $$\{e^{i_1}\wedge \cdots \wedge e^{i_k}:1\leq i_1<\cdots < i_k\leq n\}$$
    is a basis of $\Lambda^k.$
    \item [$\bullet$] For $E\subseteq \Lambda^k,$ $\spann E$ denotes the subspace spanned by $E.$
    \item [$\bullet$] Let $W$ be a subspace of $\Lambda^k$. We write $\dim W$ to denote the dimension of $W$ and $W^{\perp}$ to denote the orthogonal complement of $W.$
    \item [$\bullet$] For $b\in \Lambda^k,$ we write, identifying again $\Lambda^1$ with $\mathbb R^n,$ $$\mathbb R^n \wedge b=\Lambda^1\wedge b=\{x\wedge b:x\in \Lambda^1\}\subseteq \Lambda^{k+1}.$$
   \end{itemize}
   \item [(2)] Let $\Omega \subseteq \mathbb R^n$ be a bounded open set. 
   \begin{itemize}
       \item [$\bullet$] The spaces $C^1(\Omega; \Lambda^k),$ $W^{1,p}(\Omega;\Lambda^k)$ and $W^{1,p}_{0}(\Omega;\Lambda^k),$ $1\leq p\leq \infty$ are defined in the usual way.
       \item [$\bullet$] For $\eta \in W^{1,p}(\Omega;\Lambda^k),$ $\displaystyle \int\limits_{\Omega}\eta\ $ denotes the exterior $k$-form obtained by integrating componentwise the differential form $\eta$. Explicitly, for $1\leq i_1<\cdots < i_k\leq n,$ $$\left(\displaystyle\int\limits_{\Omega}\eta\right)_{i_1\cdots i_k}=\int\limits_{\Omega}\eta_{i_1\cdots i_k}.$$
       \item [$\bullet$] For $\eta\in W^{1,p}(\Omega;\Lambda^k),$ the exterior derivative $d\eta$ belongs to $L^p(\Omega;\Lambda^{k+1})$ and is defined by $$\left(d\eta\right)_{i_1\cdots i_{k+1}}=\displaystyle\sum\limits_{j=1}^{k+1}(-1)^{j+1}\frac{\partial \eta_{i_1\cdots i_{j-1}i_{j+1}\cdots i_{k+1}}}{\partial x_{i_j}},$$
       for $1\leq i_1<\cdots < i_{k+1}\leq n.$ If $k=0,$ then $d\eta \simeq \grad \eta$ . If $k=1,$ then for $1\leq i<j\leq n, $ $$(d\eta)_{ij}=\frac{\partial \eta_j}{\partial x_i}-\frac{\partial \eta_i}{\partial x_j}$$
       i.e., $d\eta \simeq \curl \eta.$
   \end{itemize}
   \item[(3)] For subsets $C,V \subseteq \Lambda^k,$
   \begin{itemize}
       \item [$\bullet$] $\co C$ denotes the convex hull of $C;$
       \item [$\bullet$] $\intt_{V}C$ denotes the interior of $C$ with respect to the topology relative to $V.$ 
   \end{itemize}
   \item [(4)] For a convex set $C\subseteq \Lambda^k,$ 
   \begin{itemize}
       \item [$\bullet$] $\aff C$ denotes the affine hull of $C$ which is the intersection of all affine subsets of $\Lambda^k$ containing $C;$
       \item[$\bullet$]  $\ri C$ denotes the relative interior of $C$ which is the interior of $C$ with respect to the topology relative to affine hull of $C.$ Equivalently $\ri C=\intt_{\aff C}C$; 
       \item [$\bullet$] $\rbd C$ denotes the relative boundary of $C$ which is $\overline{C}\setminus \ri C.$ 
   \end{itemize}
   \item[(5)] For a set $A\subseteq \mathbb R^n$ $\meas (A)$ denotes the Lebesgue measure of $A.$
   \item[(6)] $\mathbb R_{+}$ denotes the set of all non-negative real numbers.
\end{itemize}
\section{Non Existence of Solution: Dimensionality of E}
In this section, we will prove a non-existence result, namely, that there is no $\eta\in W_0^{1,\infty}(\Omega;\mathbb R^n)$ satisfying $$\curl \eta \in E, \textrm{ a.e. in }\Omega, \displaystyle\int\limits_{\Omega}\eta\neq 0,$$ when $\dim \spann E=n.$ The following lemma plays the main role.
\begin{lemma}\label{lemma 1'}
Let $n\in \mathbb N,$ $n\geq 4$ and let $f:\mathbb R^n \to \mathbb R^n$ be continuous  with $f(0)\neq 0$. Then $$\dim \spann \{x\wedge f(x):x\in \mathbb R^n\}\neq n.$$
\end{lemma}
\begin{remark}
    Lemma \ref{lemma 1'} is not true when $n=3.$ Let us define $f:\mathbb R^3 \to \mathbb R^3$ by $$f(x):=(e^1\otimes e^1)x+e^2, \textrm{ for all } x\in \mathbb R^3.$$ Then $\dim \spann \{x\wedge f(x):x\in \mathbb R^3\}=3.$ To see this, we note that $e^1\wedge f(e^1)=e^1\wedge e^2, e^3\wedge f(e^3)=-e^2\wedge e^3,$ and $(e^1+e^3)\wedge f(e^1+e^3)=e^1\wedge e^2-e^2\wedge e^3-e^1\wedge e^3.$ 
\end{remark}
\begin{proof}
Let us set $$\mathcal S:=\spann \{x\wedge f(x):x\in \mathbb R^n\}.$$ We prove by contradiction. Let us suppose to the contrary that $\dim \mathcal S=n$ . Note that, $\mathbb R^n \wedge f(0)\subseteq \mathcal S,$ see Proposition $2.2$ of \cite{bandyopadhyay2015some}. Furthermore, 
 we can find $x_0\in \mathbb R^n$  such that $x_0\wedge f(x_0)\wedge f(0)\neq 0 .$ Indeed, if this was  not the case, we would have  $$x\wedge f(x)\wedge f(0)=0 \textrm{ for all } x\in \mathbb R^n.$$
 Cartan's lemma, see Theorem $2.42$ of \cite{csato2011pullback}, then guarantees the existence of  $u_x\in \mathbb R^n$ satisfying $$x\wedge f(x)=u_x\wedge f(0)\textrm{ for all, }x\in \mathbb R^n,$$
 which implies that $$\mathcal S=\mathbb R^n \wedge f(0).$$
 This is a contradiction as  $\dim (\mathbb R^n\wedge f(0))=n-1$ [ Lemma $2.1$ of \cite{bandyopadhyay2015some} ], whereas $\dim \mathcal S=n.$ Therefore, we indeed have a $x_0\in \mathbb R^n\setminus \{0\}$ such that $$x_0\wedge f(x_0)\wedge f(0)\neq 0.$$
 Since $f$ is continuous at $x_0$, there exists an $\epsilon >0$ such that $$x\wedge f(x)\wedge f(0)\neq 0\textrm{ for all  }x\in B_{\epsilon}(x_0).$$
Let us find a basis  $\{a^1,\ldots, a^n\}$ of $\mathbb R^n$ inside $B_{\epsilon}(x_0)$. Then 
\begin{equation} \label{ eq 2'not n}
a^i\wedge f(a^i)\wedge f(0)\neq 0 \textrm{ for all } i=1,\ldots, n.
\end{equation}
Let us write $$\mathcal S=\left[\mathbb R^n \wedge f(0)\right ]\oplus \left[\mathbb R^n \wedge f(0)\right ]^{\perp}=\left[\mathbb R^n \wedge f(0)\right]\oplus \spann \{\omega \},$$ where $\omega \in \left[\mathbb R^n \wedge f(0)\right]^{\perp}\setminus \{0\}.$ For each $i=1,\ldots,n,$ we have 
\begin{equation}\label{02/4/2024 eq 1}
a^i\wedge f(a^i)=c_i\wedge f(0)+\beta_i\omega,
\end{equation}
for some $c_i\in \mathbb R^n$ and $\beta_i\in \mathbb R.$ Note that, thanks to equation \ref{ eq 2'not n}, we have $\beta_i\neq 0,$ for all $i=1,\ldots,n$. It follows from Equation \ref{02/4/2024 eq 1} that, for all $i=1,\ldots,n$,   $$\beta_i \  \omega \wedge f(0)\wedge a_i=0.$$
Since $\beta_i\neq 0$ for every $i=1,\ldots,n$, we have $$\omega \wedge f(0)\wedge a^i=0 \textrm{ for every } i=1, \ldots, n,$$
which implies that $\omega \wedge f(0)=0$ as $\{a^1,\ldots, a^n\}$ is a basis of $\mathbb R^n$. Using Proposition $2.16$ of \cite{csato2011pullback}, we also note that 
\begin{align*}
\langle f(0)\lrcorner \omega; x\rangle &=(-1)^{(1+1)}\langle \omega; f(0)\wedge x\rangle\\
&=0, \textrm{ for all }x\in \mathbb R^n,
\end{align*}
as $\omega \in \mathcal {S}\cap \left[\mathbb R^n \wedge f(0)\right]^{\perp}.$ It follows that $f(0)\lrcorner  \omega =0.$
This, combined with $\omega \wedge f(0)=0$ and Proposition $2.16$ of \cite{csato2011pullback} implies that 
\begin{align*}
    \|f(0)\|^2  \omega = f(0)\lrcorner \left( f(0)\wedge \omega\right) +f(0)\wedge \left(f(0)\lrcorner \ \omega\right)=0.
\end{align*}
Since $f(0)\neq 0$, we have $\omega=0$, which is a contradiction. 
Therefore $\dim \spann \{x\wedge f(x):x\in \mathbb R^n\}\neq n.$ 
\end{proof}
\begin{theorem}\label{theorem1'}
Let $n\in \mathbb N$ with $n\geq 4$, let $\Omega \subseteq \mathbb R^n$ be open, bounded and let $E\subseteq \Lambda^2(\mathbb R^n)$. Then there is no $\eta\in W_0^{1,\infty}(\Omega;\mathbb R^n)$ satisfying 
\begin{equation}\label{ eq 2'}
    \curl \eta \in E \textrm{ a.e.  in }\Omega \textrm{ and } \displaystyle \int\limits_{\Omega }\eta\neq 0
\end{equation}
 if  
 \begin{enumerate}
     \item [(i)] $\dim \spann E=n$, and 
     \item [(ii)] $\meas \{x\in \Omega: \curl \eta(x)=e\}>0$ for all $e\in E$
 \end{enumerate}
\end{theorem}
\begin{remark}
    Theorem \ref{theorem1'} is not true when $n=3,$ see Theorem $4.15$ of \cite{bandyopadhyay2007differential}. The solution $\eta$ constructed in the proof of Theorem $4.15$ of \cite{bandyopadhyay2007differential} has the property that $\displaystyle\int\limits_{\Omega}\eta \neq 0.$
\end{remark}
\begin{remark}
The case $n\leq 3$ has been done completely in \cite{bandyopadhyay2007differential}. We don't need to take the case $\int\limits_{\Omega}\eta\neq 0$ for $n\leq 3.$
\end{remark}
\begin{proof}
Let $$\mathcal P: \Lambda^2({\mathbb R^n})\to \Lambda^2(\mathbb R^n)$$
be the projection onto the orthogonal complement of $\spann E.$ Since $\eta \in W_{0}^{1,\infty}(\Omega;\mathbb R^n)$ extending $\eta $ by $0$ to $\mathbb R^n$, it follows that $$\mathcal P(\curl \eta)=0 \textrm{ a.e. in }\mathbb R^n.$$
Applying the Fourier transform, we obtain $$\mathcal P(x\wedge \hat{\eta}(x))=0 \textrm{ for all }x\in \mathbb R^n$$
which implies that $$x\wedge \hat{ \eta}(x)\in \spann E \textrm{ for all } x\in \mathbb R^n,$$
where $\hat{ \eta}(x)=\displaystyle \int\limits_{\mathbb R^n}\eta(y)\cos (2\pi \langle x;y \rangle) dy.$ Together with the Proposition $2.2$ of \cite{bandyopadhyay2015some} and the above we can conclude that  
\begin{equation}\label{eq 3'}
    \mathbb R^n \wedge \hat{ \eta}(0)\subseteq \spann \{x\wedge \hat{ \eta}(x): x\in \mathbb R^n\}\subseteq \spann E.
\end{equation}
We will now show that $$\spann \{x\wedge \hat{ \eta}(x):x\in \mathbb R^n\}=\spann E.$$
Suppose not, i.e., $\spann \{x\wedge \hat{\eta}(x): x\in \mathbb R^n\}\subsetneqq \spann E$. 
Let $m\in \spann \{x\wedge \hat{ \eta}(x):x\in \mathbb R^n\}^{\perp}$. Then 
$$\langle x\wedge \hat{\eta} (x);m\rangle =0 \textrm{ for all } x\in \mathbb R^n.$$
Using Plancherel Theorem, this implies that $$\langle \curl \eta (x);m \rangle =0\textrm{ for all } x\in \mathbb R^n \textrm{ a.e. } $$
and hence $$\langle e;m \rangle =0 \textrm{ for all }e\in E.$$  This gives us that $m\in (\spann E)^{\perp}.$ So, $\spann E\subseteq \spann \{x\wedge \hat{\eta}(x): x\in \mathbb R^n\}$.
Therefore $$\dim \spann\{x\wedge \hat{\eta}(x): x\in \mathbb R^n\}=n.$$
But it can not happen because of lemma  \ref{lemma 1'}. Thus there does not exist any solution $\eta \in W_{0}^{1,\infty}(\Omega;\mathbb R^n)$ of the problem (\ref{ eq 2'}) if $\dim \spann E=n.$
\end{proof}
\section{Restrictions on the Curl Set}
In this section, we will see that $\dim \spann E\leq n$  if we add one constraint on $E$ that $\rank [e-f]\leq 2 $ for any $e,f\in E$, where $E\subseteq \Lambda^2(\mathbb R^n)$ and $n\geq 4$.  In lemma \ref{theorem 1}, we will prove it for $n=4$ and in lemma \ref{15/12/2022 1}, we will do it for $n\geq 5.$ Let us notice here one thing that between two statements `$\rank[e-f]\leq 2 $ for any $e,f\in E$' and `$e\wedge f=0 $ for all $e,f\in E$', the later one will always imply the first one but the converse may not be true. We have given one example in remark \ref{remark 2 05/04/2024}$(ii)$ in this respect. Finally, we will establish Theorem \ref{27/3/2024 theorem 1}. Let us first prove the lemma below. 
\begin{lemma} \label{theorem 1}
Let $E\subseteq \Lambda^2(\mathbb R^4)$ be such that  $\rank [e-f]\leq 2$ for all $e,f\in E.$ Then $\dim \spann E\leq 4.$
\end{lemma}
\begin{proof}
Let $V:=\Lambda^2(\mathbb R^4).$ Let us define a bilinear map $B: V\times V \to \mathbb R$ by $$ B(u,v):=c(u\wedge v) \textrm{ for all } u,v\in V,$$
where $c(u \wedge v)\in \mathbb R$ is such that $$u\wedge v =c(u\wedge v) e^1\wedge e^2\wedge e^3 \wedge e^4.$$
Clearly, $B$ is symmetric and non-degenerate. For any subspace $F\subseteq V,$ let us define $$\tilde{F}:=\{v\in V: B(v,f)=0 \textrm{ for all }f\in F\}.$$
As $B$ is non-degenerate, the map $$B_1: V\to V^*,\textrm{ defined by }v\mapsto B(v,.)$$
is one-one and hence onto. Therefore, the map $$B_2:V\to F^*,\textrm{ defined by }v\mapsto B(v,.)$$ is also onto because  it can be written as $B_2=\psi \circ B_1,$ where $\psi : V^*\to F^*.$ Hence $$\dim V=\dim \ker B_2+\dim F^*.$$
Clearly, $\ker B_2=\tilde{F}$ and $\dim F^*=\dim F$. So 
\begin{equation} \label{eq 1}
\dim F+ \dim \tilde{F}=\dim V.
\end{equation}
Now suppose that $F\subseteq V$ is an isotropic subspace, i.e., $B|_{F\times F}=0.$ In other words, $F\subseteq \tilde{F}.$ In this case, we can say from (\ref{eq 1}) that 
\begin{align*}
\dim V&=\dim F+\dim \tilde{F}\\
&\geq \dim F+ \dim F,
\end{align*}
i.e., $\dim F\leq \frac{\dim V}{2}=3.$
Therefore, if $F\subseteq V$ is any isotropic subspace of $V,$ then $\dim F\leq 3.$

Now suppose that $E\subseteq V$ is such that  for any $e,f\in E,$ there exist $x,y\in \mathbb R^4$ such that $e-f=x\wedge y,$ i.e., $\rank[e-f]\leq 2.$ We will show that $\dim \spann E\leq 4.$

In contrary, let us suppose that $E$ contains five linearly independent elements $\zeta^0,\zeta^1,\zeta^2,\zeta^3$, $\zeta^4$ and let $$\xi^i:=\zeta^i-\zeta^0\textrm{ for }i=1,2,3,4.$$
As every $\xi^i$ has rank less than or equals to $2$, $$\xi^i\wedge \xi^i=0\textrm{ for all }i=1,2,3,4  \ \  \textrm{ [see proposition 2.37(iii) of \cite{csato2011pullback}]}$$
i.e., $$B(\xi^i,\xi^i)=0\textrm{ for all }i=1,2,3,4.$$ Now $\xi^i-\xi^j=\zeta^i-\zeta^j$ and $(\zeta^i-\zeta^j)\wedge (\zeta^i-\zeta^j)=0$ so $(\xi^i-\xi^j)\wedge (\xi^i-\xi^j)=0,$ i.e., $B(\xi^i-\xi^j, \xi^i-\xi^j)=0$ for $i,j\in \{1,2,3,4\}.$ This implies that $$\xi^i\wedge \xi^j=0\textrm{ for all }i,j\in \{1,2,3,4\}.$$
If we take $$F=\{\xi^i: i=1,2,3,4\}\textrm{ and }F':=\spann F$$
then $$B|_{F'\times F'}=0,\textrm{ i.e., } F' \textrm{ is an isotropic subspace of } V$$
and $\dim F'=4.$ It contradicts that $\dim F'\leq 3.$
Therefore, $E$ can not contain $5$ linearly independent elements. Hence $\dim \spann E\leq 4.$
\end{proof}
We will state two trivial lemmas below, the proofs of which are straightforward. We will use these lemmas \ref{ilem} and \ref{ilem2} in the proofs of lemmas \ref{25/2/2022 00} and \ref{15/12/2022 1}. 
\begin{lemma}\label{ilem}
 Let $n\in \mathbb N$ and $n\geq 3.$ Let $\{\omega,\omega'\}$ be a linearly independent subset of $\Lambda^2(\mathbb R^n)$ such that $\rank[\omega]=2$, $\rank[\omega']=2$ and $\omega \wedge \omega'=0$. Then $\dim[\ker\{\omega\}^{\perp} \cap \ker\{\omega'\}^{\perp}]=1.$  Also $\dim [\ker\{\omega\}\cap \ker \{\omega'\}]=n-3.$
\end{lemma}
\begin{proof}
As $\rank[\omega]=\rank [\omega']=2, $ let us suppose that $$\omega=x\wedge y\textrm{ and }\omega'=x'\wedge y'\textrm{ where } x,y,x',y'\in \mathbb R^n.$$
Again since $\omega\wedge \omega'=0$, the set  $\{x,y,x',y'\}$ is linearly dependent [see Theorem $2.3$,\cite{csato2011pullback}]. This gives us that $$\ker\{\omega\}^{\perp}\cap \ker \{\omega'\}^{\perp}\neq \{0\},$$
because $\ker\{\omega\}^{\perp}\cap \ker \{\omega'\}^{\perp}=\spann \{x,y\}\cap \spann \{x',y'\} =\{0\}$ implies that $\{x,y,x',y'\}$ is linearly independent, which is a contradiction. Hence $$\dim [\ker \{\omega\}^{\perp}\cap \ker \{\omega'\}^{\perp}]\geq 1.$$
Now if $\dim [\ker \{\omega\}^{\perp}\cap \ker \{\omega'\}^{\perp}]=2$ then clearly, $$\ker \{\omega\}^{\perp}\cap \ker \{\omega'\}^{\perp}=\ker \{\omega\}^{\perp}=\ker \{\omega'\}^{\perp},$$
since $\dim \ker\{\omega\}^{\perp}=\dim \ker \{\omega'\}^{\perp}=2.$ This implies that $$\spann \{x,y\}=\spann \{x',y'\},\textrm{ i.e., } \{\omega,\omega'\}\textrm{ is linearly dependent, }$$
which is a contradiction. Therefore $\dim [\ker \{\omega\}^{\perp}\cap \ker \{\omega'\}^{\perp}]=1.$

For the second part, 
\begin{align*}
    \dim [\ker \{\omega\}\cap \ker \{\omega'\}] &=n-\dim [\{\ker \{\omega\}\cap \ker \{\omega'\}\}^{\perp}]\\
    &=n-[\dim \ker\{\omega\}^{\perp}+ \ker\{\omega'\}^{\perp}]\\
    &=n-[\dim \ker\{\omega\}^{\perp}+\dim \ker\{\omega'\}^{\perp}-\dim \{\ker\{\omega\}^{\perp}\cap \ker\{\omega'\}^{\perp}\}]\\
    &=n-[2+2-1], \ \ \textrm{ as } \dim \{\ker\{\omega\}^{\perp}\cap \ker\{\omega'\}^{\perp}\}=1 \\
    &=n-3.
\end{align*}
Therefore $\dim [\ker \{\omega\}\cap \ker \{\omega'\}]=n-3.$
\end{proof}
\begin{lemma}\label{ilem2}
Let $n\in \mathbb N$ and $b\in \mathbb R^n\setminus \{0\}.$ Let $\{\omega_1,\omega_2,\ldots, \omega_m\}\subseteq \mathbb R^n\wedge b$ be a linearly independent subset, where $m\in \mathbb N, m\leq n-1$. If $\omega_i=x_i\wedge b$  for some $x_i\in \mathbb R^n,i=1,\ldots,m$ then $\{b,x_1,\ldots,x_m\}$ is linearly independent.
\end{lemma}
\begin{proof}
Let $\alpha b+\alpha_1x_1+\cdots + \alpha_m x_m=0$, where $\alpha,\alpha_1,\ldots, \alpha_m \in \mathbb R.$ Then 
\begin{align*}
0=0\wedge b&=(\alpha b+\alpha_1x_1+\cdots +\alpha_mx_m)\wedge b\\
&=\alpha_1x_1\wedge b+\cdots +\alpha_mx_m\wedge b\\
&=\alpha_1\omega_1+\cdots +\alpha_m\omega_m.
\end{align*}
As $\{\omega_1,\ldots , \omega_m\}$ is  linearly independent, $\alpha_i=0$ for $i=1,\ldots, m.$ So $\alpha b=0$ and this gives $\alpha=0.$ Therefore $\{b,x_1,\ldots, x_m\}$ is linearly independent.
\end{proof}

Now let us consider $n\geq 5.$ Using this lemma \ref{25/2/2022 00} below we will prove another main lemma \ref{15/12/2022 1} of this section which states that $\dim \spann E\leq n$ if $\rank [e-f]\leq 2$ for any $e,f\in E.$
\begin{lemma}\label{25/2/2022 00}
 Let $E\subseteq \Lambda^2(\mathbb R^n)\setminus\{0\}, \ n\geq 5$. Let $\dim \spann E=n-1$ and $\omega \wedge \omega'=0$ for any $\omega,\omega'\in E$, then 
 $$\displaystyle\bigcap_{\omega \in E}\ker \{\omega\}^{\perp}\neq \{0\}.$$
\end{lemma}
\begin{proof}
Let $\{\omega_1,\omega_2,\ldots, \omega_{n-1}\}$ be a basis of $\spann E$ and $$A_i:=\ker\{\omega_i\},\ i=1,2,\ldots, n-1.$$
Now two cases may arise:

\textit{Case 1}: In this case, let there exists $1\leq i <j<k\leq n-1$ such that $$\dim (A_i^{\perp}\cap A_j^{\perp}\cap A_k^{\perp})=1.$$
Without loss of generality, let $$\dim (A_1^{\perp}\cap A_2^{\perp}\cap A_3^{\perp})=1.$$
If $\displaystyle\bigcap\limits_{i=1}^{4}A_i^{\perp}=\{0\},\ $ then together with $\dim (A_1^{\perp}\cap A_2^{\perp}\cap A_3^{\perp})=1$ and $\omega_4\wedge \omega_4=0,$ we can say that $$(A_1^{\perp}\cap A_2^{\perp}\cap A_3^{\perp})=\spann \{b\}$$
and $$\omega_4=\alpha\wedge \beta\textrm{ with }\dim \spann\{\alpha,\beta,b\}=3\textrm{ for some }\alpha,\beta,b\in \mathbb R^n\setminus \{0\}.$$
Now let $$\omega_1=x\wedge b,\omega_2=y\wedge b,\omega_3=z\wedge b\textrm{ for some }x,y,z\in \mathbb R^n\setminus \{0\},$$ the existence of $x,y,z$ follows from Cartan's lemma, see Theorem $2.42$ of \cite{csato2011pullback}. As $\{\omega_1,\omega_2,\omega_3\}$ is linearly independent, it follows from lemma \ref{ilem2} that $$\{x,y,z,b\}\textrm{  is linearly independent} .$$
If possible, let all of $\{\alpha,\beta,b,x\},\{\alpha,\beta,b,y\}$ and $\{\alpha,\beta,b,z\}$  are linearly dependent. Then we see that $$x,y,z\in \spann \{\alpha,\beta,b\}$$
and therefore $$\spann \{x,y,z\}=\spann \{\alpha,\beta,b\}.$$ This gives $$b\in \spann \{x,y,z\},$$
which is a contradiction because $\{x,y,z,b\}$ is linearly independent. Now without loss of generality let $\{\alpha,\beta,x,b\}$ is linearly independent then clearly $$\omega_4\wedge \omega_1\neq 0$$
which contradicts our hypothesis that $$\omega_i\wedge \omega_j=0\textrm{ for all  } 1\leq i<j \leq 4.$$
Therefore $\displaystyle\bigcap\limits_{i=1}^{4}A_i^{\perp}\neq \{0\}$ and hence we can say that 
\begin{equation}\label{25/2/2022 1}
    \bigcap \limits_{i=1}^{4}\ker \{\omega_i\}^{\perp}=\bigcap \limits_{i=1}^{3}\ker \{\omega_i\}^{\perp},
\end{equation}
as $\dim \bigcap \limits_{i=1}^{3}\ker \{\omega_i\}^{\perp}=1$.
By similar argument, we can again show that 
\begin{equation}\label{25/2/2022 2}
    \bigcap \limits_{i=1,i\neq 4}^{5}\ker \{\omega_i\}^{\perp}=\bigcap \limits_{i=1}^{3}\ker \{\omega_i\}^{\perp}.
\end{equation}
From equations  $(\ref{25/2/2022 1})$ and $(\ref{25/2/2022 2})$ we can say that 
\begin{align}
    \bigcap \limits_{i=1}^{5}\ker\{\omega_i\}^{\perp}&=\left[\bigcap\limits_{i=1}^{4}\ker\{\omega_i\}^{\perp} \right]\bigcap \left[\bigcap \limits_{i=1,i \neq 2}^{5}\ker\{\omega_i\}^{\perp} \right]\nonumber\\
    &=\bigcap\limits_{i=1}^{3}\ker\{\omega_i\}^{\perp}.
\end{align}
Thus we can assert that $$\bigcap\limits_{i=1}^{3}A_i^{\perp}=A_1^{\perp}\cap A_2^{\perp}\cap A_3^{\perp}\cap A_k^{\perp} \ \textrm{ for all  } \ k\in\{4,5,\ldots, n-1\}.$$
Therefore 
\begin{align}
    \bigcap\limits_{i=1}^{n-1}A_i^{\perp} &=\bigcap\limits_{i=4}^{n-1}\left(A_1^{\perp}\cap A_2^{\perp}\cap A_3^{\perp}\cap A_i^{\perp}  \right)\\
    & = A_1^{\perp}\cap A_2^{\perp}\cap A_3^{\perp}.
\end{align}
This implies that $$\dim \left( \bigcap \limits_{i=1}^{n-1}A_i^{\perp}\right)=1.$$
Therefore $$\bigcap\limits_{i=1}^{n-1}A_i^{\perp}=\spann\{b\}\textrm{  for some } \ b\in \mathbb R^n\setminus \{0\}$$ and hence $$\spann E=\spann \{\omega_1,\omega_2,\ldots, \omega_{n-1}\}=\mathbb R^n \wedge b.$$

\textit{Case 2}: In this case  we let $$A_i^{\perp}\cap A_j^{\perp}\cap A_k^{\perp}=\{0\}\textrm{ for any } \ 1\leq i<j<k\leq n-1.$$
As $n\geq 5$, there exists $l\in \{1,2,\ldots, n-1\}\setminus \{i,j,k\}$ such that  $$\dim  A_l^{\perp}=2\textrm{  and } \ A_i^{\perp}\cap A_j^{\perp}\cap A_k^{\perp}\subseteq A_l^{\perp}.$$
This gives us  $\dim (A_l^{\perp}+(A_i^{\perp}\cap A_j^{\perp}\cap A_k^{\perp}))=2$ and so 
\begin{equation}\label{25/2/2022 3}
    \dim(A_l\cap (A_i+A_j+A_k))=n-2.
\end{equation}
Because of $A_i^{\perp}\cap A_j^{\perp}\cap A_k^{\perp}=\{0\}$, $$A_i^{\perp}\cap A_j^{\perp}\cap A_k^{\perp}\cap A_l^{\perp}=\{0\}$$ and it follows that $$\dim (A_i+A_j+A_k+A_l)=n.$$
Now 
\begin{align*}
    \dim(A_i+A_j+A_k+A_l) &=\dim A_i +\dim A_j +\dim A_k +\dim A_l\\
    &\ \ - \dim(A_j\cap A_k)-\dim(A_j\cap A_l)-\dim (A_k\cap A_l)\\
    &\ \ +\dim(A_j\cap A_k\cap A_l)-\dim (A_i\cap (A_j+A_k+A_l))\\
    &= (n-2)\times 4-(n-3) \times 3 + \dim (A_j\cap A_k\cap A_l)-(n-2).\\
    & \ \ \hspace{1 cm}  [\textrm{using lemma \ref{ilem} and equation \ref{25/2/2022 3}}]\\
    &=\dim (A_j\cap A_k\cap A_l).
\end{align*}
That is, $\dim(A_j\cap A_k\cap A_l)=n$. But $n\geq 5$, so $\dim (A_j\cap A_k\cap A_l)\geq 5$. Also  $$A_j\cap A_k\cap A_l\subseteq A_j\cap A_k$$
and $\dim(A_j\cap A_k)=n-3$ by lemma \ref{ilem}. So $\dim(A_i\cap A_j\cap A_k)$ can not be equal to $n$ for $n\geq 5.$ Thus we are getting contradiction and it follows that case 2 can not happen.
\end{proof}

\begin{remark}
\ \ 
\begin{itemize}
\item[(i)] We can not apply the proof of the above lemma \ref{25/2/2022 00} for $\Lambda^2(\mathbb R^4)$, i.e., if $E\subset \Lambda^2(\mathbb R^4)$ such that $\dim \spann E=3$ and $\omega\wedge \omega'=0$ for any $\omega,\omega' \in E$ then it may not happen that $\spann E =\mathbb R^4\wedge b$ for some $b\in \mathbb R^4\setminus \{0\}$ because if we take the set $E$ as $\{e^1\wedge e^2, e^1\wedge e^3,e^2\wedge e^3\}$ then $E$ can not be written as a subset of $\mathbb R^4\wedge b$ for any $b\in \mathbb R^4\setminus \{0\}.$ Importantly, case $2$ of the above lemma is true for this example. 

\item [(ii)] \label{remark 2 05/04/2024} The aforementioned lemma is not true if we replace the case $\omega \wedge \omega'=0$ for all $\omega, \omega'\in E$ with $\rank [\omega -\omega']\leq 2$ for any $\omega,\omega' \in E.$ For example, if we take $E$ as $$E=\{e^2\wedge e^3, e^2\wedge e^3 + e^1\wedge e^2,e^2\wedge e^3+ e^1\wedge e^3,\ldots, e^2\wedge e^3 + e^1\wedge e^{n-1}\}$$ for $n\in \mathbb N$ and $n\geq 5,$ then $\dim \spann E=n-1$ but $\spann E$ can not be written as $\mathbb R^n \wedge b$ for any $b\in \mathbb R^n.$ For $n=4,$ we can simply take the set $\{e^2\wedge e^3,e^2\wedge e^3+e^1\wedge e^2, e^2\wedge e^3+ e^1\wedge e^4 \}$.
\end{itemize}
\end{remark}
We will use the following lemma in Theorem \ref{27/3/2024 theorem 1}.
\begin{lemma} \label{15/12/2022 1}
Let $n\in \mathbb N, n\geq 5$ and $E\subset \Lambda^2(\mathbb R^n)$ be such that $\rank [e-f]\leq 2$ for any $e,f\in E.$ Then $\dim \spann E\leq n.$


\end{lemma}
\begin{remark}
    For $n=4,$ we have done separate proof of this lemma in \ref{theorem 1} because we will use lemma \ref{25/2/2022 00} in the proof below which may not hold for $n=4.$
\end{remark}
\begin{proof}
We will show that $\dim \spann E\leq n$ if we take $\rank[e-f]\leq 2$ for any $e,f\in E.$ Let us suppose that $E$ contains $n+1$ linearly independent elements $\omega_0,\omega_1,\ldots, \omega_n.$ Let $$\psi_i:=\omega_i-\omega_0,\  i=1,\ldots, n.$$
The set $\{\psi_i:i=1,\ldots, n\}$ is linearly independent set with $$\psi_i\wedge \psi_j=0\textrm{ for any } i,j.$$
Indeed, $\psi_i-\psi_j=\omega_i-\omega_j$ and $(\omega_i-\omega_j)\wedge (\omega_i-\omega_j)=0$, i.e., $(\psi_i-\psi_j)\wedge (\psi_i-\psi_j)=0$. 
Therefore from lemma \ref{25/2/2022 00} we can say that there exists $b\in \mathbb R^n \setminus\{0\}$
such that $$\spann \{\psi_1,\ldots,\psi_{n-1}\}=\mathbb R^n\wedge b.$$
Similarly there exists $b'\in \mathbb R^n \setminus \{0\}$ 
such that $$\spann \{\psi_2,\ldots, \psi_n\}=\mathbb R^n \wedge b'.$$
As $n\geq 5$, we can say from lemma \ref{ilem} that $$\ker\{\psi_2\}^{\perp}\cap \ker \{\psi_3\}^{\perp}=\displaystyle\bigcap_{i=1}^{n-1}\ker \{\psi_i\}^{\perp}$$
and $$\ker \{\psi_2\}^{\perp}\cap \ker \{\psi_3\}^{\perp}=\displaystyle\bigcap_{i=2}^{n}\ker\{\psi_i\}^{\perp}$$
as $\dim [\ker \{\psi_2\}^{\perp}\cap \ker \{\psi_3\}^{\perp}]=1.$  Therefore $$\mathbb R^n\wedge b=\spann \{\psi_1,\ldots,\psi_{n-1}\}=\mathbb R^n\wedge b'=\spann \{\psi_2,\ldots, \psi_n\},$$
which is a contradiction because $\{\psi_1,\ldots, \psi_n\}$ is a linearly independent set. 

Hence $\dim \spann E\leq n.$
\end{proof}
Let us prove the main theorem on differential inclusions of this section using the previous lemma   \ref{15/12/2022 1} and Theorem \ref{theorem1'}.
\begin{theorem} \label{27/3/2024 theorem 1}
    Let $n\in \mathbb N, n\geq 5$ and $E\subseteq \Lambda^2(\mathbb R^n)$ be such that $$\omega \wedge \omega'=0 \textrm{ for all }\ \omega ,\omega'\in E.$$
    Let $\Omega \subseteq \mathbb R^n$ be  an open, bounded set.  Then there exists $\eta \in W_0^{1,\infty}(\Omega;\mathbb R^n)$ such that
    \begin{align} \label{16/12/2022 1}
     & \curl \eta \in E \textrm{ a.e. in } \Omega  \nonumber \\
      & \meas\{x\in \Omega : \curl \eta (x)=e\}>0 \textrm{ for all } e\in E\\
      & \textrm{ and } \int\limits_{\Omega}\eta \neq 0, \nonumber
     \end{align}
     if and only if $0\in \rico E$ and $\dim \spann E=n-1.$
\end{theorem}
 \begin{proof}
     If there exists $\eta \in W_0^{1,\infty}(\Omega;\mathbb R^n)$ such that 
     \begin{align} \label{15/12/2022 2}
     & \curl \eta \in E \textrm{ a.e. in } \Omega  \nonumber \\
      & \meas\{x\in \Omega : \curl \eta (x)=e\}>0 \textrm{ for all } e\in E\\
      & \textrm{ and } \int\limits_{\Omega}\eta \neq 0, \nonumber
     \end{align}
     then from Theorem $2.5$ of \cite{bandyopadhyay2015some} we can say that $$\dim \spann E\geq n-1.$$
     As $E$ has the property that $\omega \wedge \omega' =0$ for any $\omega,\omega' \in E,$ it implies that $$\rank[\omega-\omega']\leq 2 \textrm{ for any } \omega, \omega'\in E.$$
     Hence  from lemma \ref{15/12/2022 1} it follows that $$\dim \spann E\leq n.$$
     Now for $\dim \spann E=n,$ there does not exist any solution $\eta \in W_0^{1,\infty}(\Omega;\mathbb R^n)$ of the problem (\ref{15/12/2022 2}) which directly follows from Theorem \ref{theorem1'}. Hence $\dim \spann E=n-1.$ Using lemma $2.4$ of \cite{bandyopadhyay2015some}, it holds that $0\in \rico E.$ 

     Conversely, if $\dim \spann E=n-1,$ then using lemma \ref{25/2/2022 00} there exists $b\in \mathbb R^n \setminus \{0\}$ such that $\spann E=\mathbb R^n \wedge b$. Now a solution $\eta $ exists satisfying (\ref{16/12/2022 1}) from corollary $3.9$ of \cite{bandyopadhyay2015some}.
     
 \end{proof}

\begin{remark}
    For $n=4,$ the necessary part of the above theorem is true and it follows from Theorem \ref{theorem1'} and lemma \ref{theorem 1}. But the converse part is not true because $\spann E$ may not be written as $\mathbb R^4\wedge b$ for some $b\in \mathbb R^4\setminus \{0\}$ always.
\end{remark}
\section{Existence of Solution: Dimension 2n-3}
\label{27/3/2024 section} In this section, we will see one existence result of solution $\eta \in W_0^{1,\infty}(\Omega;\mathbb R^n)$ at dimension $(2n-3)$ for the following differential inclusion problem 
\begin{align*}
&\curl \eta\in E \textrm{ a.e. in }\Omega,\\
&\meas\{x\in \Omega:\curl \eta (x)=e\}>0 \textrm{ for all } e\in E.
\end{align*}
Let $n\in \mathbb N, n\geq 4.$ Let $E\subseteq \Lambda^2(\mathbb R^n)\setminus \{0\}$ and $\dim \spann E=2n-3$ where 
$E= \{e^1\wedge e^2, e^1\wedge e^3,\ldots, e^1\wedge e^n, -e^1\wedge e^2,-e^1\wedge e^3,\ldots, -e^1\wedge e^n,  e^2\wedge e^3,e^2\wedge e^4,\ldots, e^2\wedge e^n, -e^2\wedge e^3,-e^2\wedge e^4,\ldots, -e^2\wedge e^n\}.$
Clearly, $\spann E=\mathbb R^n\wedge e^1+\mathbb R^n\wedge e^2.$ Let us write $E=E_1\cup E_2$, where $$E_1=\{e^1\wedge e^2,e^1\wedge e^3,\ldots, e^1\wedge e^n, -e^1\wedge e^2,-e^1\wedge e^3,\ldots, -e^1\wedge e^n\}$$ and $$E_2=\{e^2\wedge e^1,e^2\wedge e^3,\ldots , e^2\wedge e^n, -e^2\wedge e^1,-e^2\wedge e^3,\ldots , -e^2\wedge e^n\}.$$ 
Then, $\spann E_1=\mathbb R^n\wedge e^1$, $\spann E_2=\mathbb R^n\wedge e^2$ and $0\in \rico E_1\cap \rico E_2$. 

\textbf{Step-1:} Let $G=I_1\times I_2\times \cdots \times I_n$ be  open unit cube in $\mathbb R^n$, where $I_i=(0,1)$ for each $i=1,\ldots, n.$ Let us divide the domain $G$ into two parts as follows: $$G_1=\left(0,\frac{1}{2}\right)\times I_2\times \cdots \times I_n$$ and $$G_2=\left(\frac{1}{2},1\right)\times I_2\times \cdots \times I_n.$$ 
Clearly, $G_1,G_2$ are open, bounded sets in $\mathbb R^n.$ Then there exist $\overline{\eta_1}\in W_0^{1,\infty}(G_1;\mathbb R^n)$ and $\overline{\eta_2}\in W_0^{1,\infty}(G_2;\mathbb R^n)$ such that $$\curl \overline{\eta_1}\in E_1 \textrm{ a.e. in } G_1$$ and $$\curl \overline{\eta_2} \in E_2 \textrm{ a.e. in }G_2.$$
Let us define a mapping $\overline{\eta} \in W_0^{1,\infty}(G;\mathbb R^n)$ by 
\[
\overline{\eta}(x)=
\begin{cases}
 &\overline{\eta_1}(x), \ \ \textrm{if } x\in G_1 \textrm{ a.e. }\\
 &\overline{\eta_2}(x) , \ \ \textrm{if }x\in G_2 \textrm{ a.e. }
\end{cases}
\]
Then $$\curl \overline{\eta} \in E \textrm{ a.e. in }G,$$
i.e., there exists $\overline{\eta}\in W_0^{1,\infty}(G;\mathbb R^n)$ such that 
\begin{align*}
&\curl\overline{ \eta} \in E \textrm{ a.e. in } G,\\
&\meas\{x\in \Omega:\curl\overline{\eta} (x)=e\}>0 \textrm{ for all } e\in E,
\end{align*}
where $\dim \spann E=2n-3$ and $\spann E=\mathbb R^n \wedge e^1+\mathbb R^n \wedge e^2.$

\textbf{Step 2:} Let $\Omega \subseteq \mathbb R^n$ be an open, bounded set. Using Vitali's covering theorem, there exists a sequence $\{G_k: k\in \mathbb N\}$, where $G_k$'s are translated and dilated sets of $\overline{G}$ (closure of $G$) and $G$ is as defined in step-$1$ above such that 
\begin{align*}
    &G_k\subseteq \Omega \textrm{ for each } k\in \mathbb N,\\
   & G_h\cap G_k =\emptyset \textrm{ for all } h,k\in \mathbb N, h\neq k,\\
   & \meas \left( \Omega \setminus \bigcup_{k\in \mathbb N}G_k\right)=0.
\end{align*}
Let $G_k:= a_k+t_k \overline{G}$, where $a_k\in \mathbb R^n$ and $t_k\in \mathbb R\setminus \{0\}$. Let us define a map $\eta \in W_0^{1,\infty}(\Omega;\mathbb R^n)$ by 
\[
\eta(z)=
\begin{cases}
    \overline{\eta}\left(\frac{z-a_k}{t_k}\right), \ \ \textrm{ if } z\in G_k\\
    0, \ \ \textrm{ if } z\in \Omega \setminus \displaystyle\bigcup_{k\in \mathbb N}G_k .
\end{cases}
\]
Clearly, $\eta \in W_0^{1,\infty}(\Omega;\mathbb R^n)$ and $$\curl \eta \in E \textrm{ a.e. in }\Omega.$$
Thus we are getting a solution $\eta \in W_0^{1,\infty}(\Omega;\mathbb R^n)$ for the following differential inclusion problem 
\begin{align*}
&\curl \eta \in E \textrm{ a.e. in } \Omega,\\
& \meas \{z\in \Omega : \curl \eta(z)=e\}>0 \textrm{ for all } e\in E.
\end{align*}
\section{The Case of a k-form}
In this section, we will generalize lemma \ref{lemma 1'} to exterior $k$-form.
\begin{theorem}\label{06/04/2024 lemma 1}
Let $n\in \mathbb N,$ $1\leq k \leq n-3.$ Suppose $f:\mathbb R^n\to \Lambda^k(\mathbb R^n)$ be  continuous such that $f(0)\neq 0$ and $f(0)$ is $k$-divisible, i.e., $f(0)=c^1\wedge \ldots \wedge c^k$ for some $c^i\in \mathbb R^n\setminus \{0\}$ for  $i=1,\ldots, k.$ Then 
\begin{equation}\nonumber
\dim \spann \{x\wedge f(x): x\in \mathbb R^n\}\neq n-k+1.
\end{equation}

\end{theorem}

\begin{proof} 
Let $$\mathcal S:=\spann\{x\wedge f(x):x\in \mathbb R^n \}.$$
Let us suppose to the contrary that $\dim \mathcal S=n-k+1.$  We know that 
\begin{equation}\label{260920221}
\mathbb R^n \wedge f(0)\subseteq \mathcal S,
\end{equation}
from proposition  $2.2$ of \cite{bandyopadhyay2015some}. Therefore $$\mathcal S=\left(\mathbb R^n \wedge f(0)\right)\oplus \spann (\omega ),$$ 
for some $\omega \in \Lambda^{k+1}(\mathbb R^n)\setminus \{0\}$.  
Using proposition $2.16$ of \cite{csato2011pullback}, we can write
\begin{align*}
\langle f(0)\lrcorner \omega; x\rangle &=(-1)^{k+1}\langle \omega; f(0)\wedge x\rangle\\
&=0, \textrm{ because }\omega \in \left[\mathbb R^n\wedge f(0)\right]^{\perp}\textrm{ for all }x\in \mathbb R^n.
\end{align*}
Therefore 
\begin{equation}\label{260920225}
f(0)\lrcorner \  \omega =0.
\end{equation}
Let us choose $x^0\in \mathbb R^n$ such that $x^0\wedge f(x^0)\in \Lambda^{k+1}\setminus (\mathbb R^n\wedge f(0))$. Such an $x^0$ exists because $\mathbb R^n \wedge f(0)\subsetneqq \mathcal S$.
Since $\Lambda^{k+1}\setminus (\mathbb R^n \wedge f(0))$ is open, it follows from the continuity of $f$ that, for some $\epsilon >0$ $$x\wedge f(x)\in \Lambda^{k+1}\setminus (\mathbb R^n \wedge f(0)), $$
for all $x\in B(x^0,\epsilon).$\\
\\
Let us choose a basis $\{a^1,\ldots,a^n\}$ of $\mathbb R^n$ within $B(x^0,\epsilon)$. Then for each $j\in \{1,2,\ldots,n\}$, we find $\beta^j\in \mathbb R$ and $b^j\in \mathbb R^n$ such that $$a^j\wedge f(a^j)=b^j\wedge f(0)+\beta^j\omega.$$
Note that, $\beta^j\neq 0 \textrm{ for all }j\in \{1,2,\ldots,n\}.$\\
\\
We have $f(0)=c^1\wedge \cdots \wedge c^k$. Let $1\leq r\leq k$ be fixed. Then for all $j\in \{1,\ldots, n\}$, $$\beta^j (\omega \wedge c^r\wedge a^j)=0$$ 
implies that $$(\omega \wedge c^r)\wedge a^j=0 \textrm{ for all }j=1,\ldots, n.$$
This gives us $$\omega \wedge c^r=0 \textrm{ for all }r\in \{1,\ldots ,k\}.$$
Therefore, we have $$\omega =(c^1\wedge \ldots \wedge c^k)\wedge \omega'$$
for some $\omega'\in \mathbb R^n,$ where $c^i\lrcorner\omega'=0$ for all $i=1,\ldots, k.$ \\
\\
Since $(c^1\wedge \ldots \wedge c^k)\lrcorner \omega =0$ [from equation (\ref{260920225})], we have $\omega'=0$. \\
\\
Hence $\omega=0$, a contradiction. This proves the theorem.
\end{proof}

\begin{remark} \label{27/3/2024 remark 1}
For $k=n-2$ and $n\geq 4$, there exists a continuous function such that the above Theorem \ref{06/04/2024 lemma 1} fails. For example, let us take $f:\mathbb R^n \to \Lambda^{n-2}(\mathbb R^n)$ defined by $$f(x_1,\ldots, x_n)=(x_1+1)e^1\wedge \cdots \wedge e^{n-2}+x_2 \ e^1\wedge \cdots \wedge e^{n-3}\wedge e^{n-1}, $$
for all $x=(x_1,\ldots, x_n)\in \mathbb R^n$. Then 
\begin{align*}
    x\wedge f(x)&=(x_1 e^1+\cdots +x_n e^n)\wedge f(x)\\
    & =\left((-1)^{n-3}x_{n-2}x_2 +(-1)^{n-2}x_{n-1}(x_1+1) \right) e^1\wedge\cdots \wedge e^{n-1}\\
    & \ + (-1)^{n-2}x_n (x_1+1)e^1\wedge \cdots \wedge e^{n-2}\wedge e^n \\
    & \ + (-1)^{n-2} x_n x_2 e^1\wedge \cdots \wedge e^{n-3}\wedge e^{n-1}\wedge e^n.
\end{align*}
Clearly, $\dim \spann \{x\wedge f(x): x\in \mathbb R^n\}$ has dimension $n-(n-2)+1=3.$

For $k=n-1,$ the above Theorem \ref{06/04/2024 lemma 1} is trivially true.
\end{remark}

\section*{Funding}
The research of N. Nesha was supported by CSIR Senior Research Fellowship [File No:09/921(0309)/2020-EMR-I].

\bibliographystyle{plain}
\bibliography{8_mybib}

\end{document}